\documentclass{amsart}

\usepackage[usenames]{color}
\usepackage{amssymb}
\usepackage{graphicx}
\usepackage{amscd}
\usepackage{color}
\usepackage{fullpage}
\usepackage{float}
\usepackage{graphics,amsmath,amssymb}
\usepackage{amsthm}
\usepackage{amsfonts}
\usepackage{latexsym}
\usepackage{epsf}
\newtheorem{theorem}{Theorem}

\newtheorem{proposition}[theorem]{Proposition}

\newtheorem{defin}[theorem]{Definition}

\newtheorem{examp}[theorem]{Example}

\newtheorem{rema}[theorem]{Remark}

\title{A Counting Function}
\author{Milan Janji\'c and Boris Petkovi\'c}
\address{Department of Mathematics and Informatics\\
 University of Banja Luka \\
Republic of Srpska, BA}
\begin{document}
\maketitle

\begin{abstract}
We define a counting function that is  related to the
binomial coefficients.
 An explicit formula for this function   is  proved.
 In some particular cases, simpler explicit formulae are derived. We also derive a
  formula for the number of $(0,1)$-matrices, having a fixed number of $1'$s, and having no
  zero rows and zero columns.
 Further, we  show that our function satisfies several recurrence relations.

 The relationship of our counting function with different classes of integers is then  examined.
These classes include: different kind of figurate numbers,
  the number of points on the surface of a square pyramid,
  the magic constants, the truncated square numbers,
  the coefficients of the Chebyshev polynomials,
  the Catalan numbers,
  the Dellanoy numbers,  the Sulanke numbers, the numbers of the   coordination sequences, and the number  of the crystal ball sequences of a cubic lattice.

In the last part of the paper, we prove that several configurations are counted by our function.
Some of these are: the number of spanning subgraphs of the complete bipartite graph,
the number of square  containing in a square,  the number of colorings of points on a line,
 the number of divisors of some particular numbers, the number of all parts in the compositions
  of an integer,
 the numbers of the weak compositions of integers, and the number of particular lattice paths.
 We conclude  by counting
 the number of possible moves of the rook, bishop, and queen on a chessboard.

The most statements in the paper are provided by bijective proofs in terms of  insets, which are defined in the paper. With this we want to show that different configurations may be counted by the same method.
\end{abstract}



\section{Introduction}

  For a set  $Q=\{q_1,q_2,\ldots,q_n\}$  of positive integers and  a nonnegative integer  $m,$
 we consider  the set $X$ consisting of $n$  blocks
$X_i,\;(i=1,2,\ldots,n),$ $X_i$ having $q_i$ elements, and a block
$Y$ with $m$ elements. We call $X_i$ the main blocks, and $Y$ the additional block of $X.$

\begin{defin}By an $(n+k)$-inset of $X,$ we shall mean an $(n+k)$-subset of $X,$
intersecting each main block. We let $m,n\choose k,Q$  denote the
number of $(n+k)$-insets of $X.$
\end{defin}
In all what follows $m,n,k,Q$ will have the meaning as in the
preceding definition.
 Also, elements of insets, lying either in the same main block or in the additional block, will always be written by increasing indices.
\begin{rema} Note that this function is first defined in Janji\'c paper \cite{MJ}.
\end{rema}

The case $n=0$ also may be considered. Then, there are no main
blocks, so that ${m,0\choose k,\emptyset }={m\choose k}.$ Also,
when each main block has only one element,
  we have  \[{m,n\choose k,Q}= {m\choose k}.\]
Hence, the function $m,n\choose k,Q$  is a
 generalization of the binomial coefficients.

In the case $k=0,$ we obviously have \[{m,n\choose 0,Q}=q_1\cdot
q_2\cdots q_n.\] Thus,  the product  function is a particular case
of our function.

Note that, when $q_1=q_2=\ldots=q_n=q,$ we write ${m,n\choose k,q}$
instead of ${m,n\choose k,Q}.$ In this case, we have
\[{m,n\choose 0,q}=q^n.\]
Some powers may be obtained in a less obvious way.
\begin{proposition} The following formula holds
\[{m,1\choose 2,2}=m^2.\]
\end{proposition}
\begin{proof}
Let $\{x_1,x_2\}$ be the main  and $\{y_1,y_2,\ldots,y_m\}$ be the additional block of $X.$  It is enough to
define a bijection between $3$-insets of $X$ and the set of
$2$-tuples $(s,t),$ where $s,t\in [m].$ A bijection goes as
follows:

\[\begin{array}{cc}
1.&\{x_1,y_i,y_j\}\leftrightarrow (i,j),\\
2.&\{x_2,y_i,y_j\}\leftrightarrow (j,i),\\
3.&\{x_1,x_2,y_i\}\leftrightarrow (i,i).
\end{array}\]
\end{proof}
\begin{proposition} The following formula is true \[{1,2\choose 1,q}=q^3.\]
\end{proposition}
\begin{proof}  Let $X_i=\{x_{i1},x_{i2},\ldots, x_{iq}\},\;(i=1,2)$ be the main blocks, and  $Y=\{y\}$  the additional block of $X.$ We need a bijection
 between $3$-insets of $X$ and $3$-tuples $(s,r,t),$ where $r,s,t\in[q].$ A bijection is defined in the following way:
 \[\begin{array}{cc}
 1.&\{x_{1s},x_{1t},x_{2r}\}\leftrightarrow (s,t,r),\\
 2.&\{x_{1s},x_{2t},x_{2r}\}\leftrightarrow (r,t,s),\\
 3.&\{x_{1s},x_{2t},y\}\leftrightarrow (s,s,t).
 \end{array}\]
\end{proof}
\begin{proposition} The following formula is true:
 \begin{equation}\label{pr3}{0,n\choose k,2}=2^{n-k}{n\choose k}.\end{equation}
 \end{proposition}
\begin{proof} We obtain ${0,n\choose k,2}$ by choosing  both elements from arbitrary
$k$ main blocks, which may be done in ${n\choose k}$ ways, and
one element from each of the remaining $n-k$ main
blocks, which may be done in $2^{n-k}$ ways.
\end{proof}

The particular case $m=0$ may be interpreted as
numbers of $1$'s in a $(0,1)$-matrix. The following proposition is
obvious:
\begin{proposition}\label{pp1} The number ${0,n\choose k,q}$ equals the number  of $(0,1)$-matrices of order $q\times n$  containing   $n+k$ $1$'s, and which have no zero columns. \end{proposition}

Now, we  count the number of $(0,1)$-matrices with a fixed number
of $1$'s which have no zero rows and zero columns. Let $M(n,k,q)$
denote the number of such matrices of order $q\times n,$  which
have  $n+k$ $1$'s.

\begin{proposition} The following formula is true: \begin{equation}
\label{pr2} M(n,k,q)=\sum_{i=0}^{q}(-1)^{q+i}{q\choose
i}{0,n\choose k,i},\;(q>1).\end{equation} \end{proposition}
\begin{proof} According to (\ref{pr3}), we have
${0,n\choose k,q}$ $(0,1)$-matrices, which have $n+k$ $1$'s and
 no zero columns. Among them, there are ${q\choose
i}M(n,k,q-i),\;(i=0,1,2,\ldots,q)$ matrices having exactly $i$
zero rows. It follows that
\[{0,n\choose k,q}=\sum_{i=0}^{q}{q\choose i}M(n,k,q-i),\] and
the proof follows from the inversion formula. \end{proof}

Obviously, the function $M(n,k,q)$  has the
property: \[M(n,k,q)=M(q,n+k-q,n).\]

Using (\ref{pr3}) and (\ref{pr2}), we obtain the
 binomial identity:
\[{n\choose k}=\frac{1}{2^{n-k}}\sum_{i=0}^n(-1)^{n+i}{n\choose
i}{0,\;2\choose n+k-2,\;i},\;(k>0).\]
\section{Explicit Formulae and Recurrences}
We first consider the particular case $m=k=1,$ when the explicit
formula for our function is easy to derive.
\begin{proposition} The following equation holds:
\begin{equation}\label{1q}{1,n\choose 1,Q}=q_1q_2\cdots q_n\left(\frac{\sum_{i=1}^nq_i-n+2}{2}\right).\end{equation}
In particular, we have
\begin{equation}\label{2q}{1,2\choose 1,Q}=\frac{q_1q_2(q_1+q_2)}{2}.\end{equation}
\end{proposition}

\begin{proof}
 If the element of the additional block is inserted into an $(n+1)$-inset, then each of the remaining elements must be chosen from different main blocks. For this, we have $q_1\cdot q_2\cdots q_n$ possibilities.
 If it is not inserted, we take two elements from one of the main blocks and one element from each of the remaining main blocks.
 For this, we have $\sum_{i=1}^n{q_i\choose 2}q_1\cdots q_{i-1}q_{i+1}\cdots q_n$ possibilities. All in all, we have
$q_1q_2\cdots q_n\left(\frac{\sum_{i=1}^nq_i-n+2}{2}\right)$
possibilities.
\end{proof}

Using the inclusion-exclusion principle, we  derive an explicit
formula for $m,n\choose k,Q$.
\begin{proposition} The following formula is true:
 \begin{equation}\label{osfor}{m,n\choose k,Q}=\sum_{I\subseteq
[n]}(-1)^{|I|}{|X|-\sum_{i\in I}q_i\choose n+k},\end{equation}
where the sum is taken over all subsets of $[n].$
\end{proposition}

\begin{proof} For $i=1,2,\ldots,n$
 and an $n+k$-subset $Z$ of $X,$ we define the following  property:

\begin{center} The block $X_i$ does not intersect $Z.$\end{center}

Using the PIE method, we obtain \[{m,n\choose k,Q}=\sum_{I\subseteq
[n]}(-1)^{|I|}N(I),\] where $N(I)$ is the number of
$(n+k)$-subsets of $X,$ which do not intersect main blocks
$X_i,\;(i\in I).$ It is clear that there are
\[{|X|-\sum_{i\in I}q_i\choose n+k}\] such subsets, and the formula is proved.
\end{proof}
In the particular cases $n=1$ and $n=2,$ we obtain the following
formulae:
\begin{equation}\label{sn=1}{m,1\choose k,q}={q+m\choose k+1}-{m\choose k+1},\end{equation}

\begin{equation}\label{ndva}
{m,2\choose k,Q}={q_1+q_2+m\choose k+2}-{q_1+m\choose
k+2}-{q_2+m\choose k+2}+{m\choose k+2}.\end{equation}

In the case $q_1=q_2=\cdots=q_n=q,$  formula (\ref{osfor}) takes
a simpler form: \begin{equation}\label{f1} {m,n\choose
k,q}=\sum_{i=0}^n(-1)^i{n\choose i}{nq+m-iq\choose
n+k}.\end{equation}

If $q=1,$ then ${m,n\choose k,1}={m\choose k},$ so that formula
(\ref{f1}) implies the  well-known binomial identity \[{m\choose
k}=\sum_{i=0}^n(-1)^i{n\choose i}{n+m-i\choose n+k}.\]  Next,
since we have  ${m,n\choose 0,q}=q^n,$  equation (\ref{f1}) yields
\[q^n=\sum_{i=0}^n(-1)^i{n\choose i}{qn+m-qi\choose n}.\]

 Note that the left hand side does not depend
on $m,$ so we have here a family of identities.

We next derive a recurrence relation which stresses the
similarity of our function and the binomial coefficients.

\begin{proposition}The following formula holds: \begin{equation}\label{rek2}{m+1,n\choose k+1,Q }={m,n\choose
k+1,Q}+{m,n\choose k,Q}.\end{equation} \end{proposition}
\begin{proof} Let $Y=\{y_1,y_2,\ldots,y_m,y_{m+1}\}$ be the
additional block of $X.$ We divide all $(n+k+1)$-insets of $X$
into two classes. In the first class are the insets which do not
contain the element $y_{m+1}.$ There are ${m,n\choose k+1,Q}$ such
insets. The second class consists  of the remaining insets,
namely those that contain $y_{m+1}.$ There are ${m,n\choose
k,Q}$ such insets.
\end{proof}

The next formula reduces the case of arbitrary $m$ to the case
$m=0.$

\begin{proposition} The following formula is true: \begin{equation}\label{prm}{m,n\choose k,Q}=\sum_{i=0}^{m}{m\choose i}{0,n\choose
k-i,Q}.\end{equation} \end{proposition} \begin{proof}
 We may obtain all $(n+k)$-insets of $X$ in the following way:
 \begin{enumerate}
 \item There are ${0,n\choose k,Q}$ $(n+k)$-insets not containing elements from $Y.$
 \item
  The remaining $(n+k)$-insets of $X,$  are a union of  some $(n+k-i)$-inset of $X,$ not intersecting $Y,$ and some
$i$-subset  of $Y,$  where $1\leq i\leq m.$ There are ${m\choose
i}$ such insets.
\end{enumerate}
\end{proof}

Particularly, we have \[{m,1\choose k,q}=\sum_{i=0}^{m}{m\choose
i}{0,1\choose k-i,q}.\] According to (\ref{sn=1}), we have
\[{0,1\choose k-i,q}={q\choose k-i+1}-{0\choose k-i+1},\;{m,1\choose k,q}={m+q\choose k+1}-{m\choose k+1}.\] As a consequence,  we obtain the Vandermonde convolution:
          \[{q+m\choose
k+1}=\sum_{i=0}^{m}{m\choose i}{q\choose k+1-i}.\]

Using   (\ref{pr3}) and (\ref{prm}), we obtain  another
explicit formula for ${m,n\choose k,2}:$

\begin{equation}\label{q=2}{m,n\choose k,2}=2^{n-k}\sum_{i=0}^m2^{i}{m\choose i}{n\choose k-i}.\end{equation}
Finally, we derive two recurrence relations with respect to the
number of main blocks:
\begin{proposition} Let $j\in [n]$ be arbitrary. Then,
\begin{equation}\label{krk}{m,n\choose
k,Q}=\sum_{i=0}^{q_j-1}{m+i,n-1\choose
k,Q\setminus\{q_j\}}.\end{equation}

\begin{equation}\label{r3}{m,n\choose k,Q}=\sum_{i=1}^{q_j}{q_j\choose i}{m,\;n-1\choose k-i+1, Q\setminus\{q_j\}},\end{equation}
\end{proposition}
\begin{proof}

Take $x_{jt}\in X_j$ arbitrarily. Consider the set $Z_j,$ the main
blocks of which are all the main blocks of $X,$ except $X_j.$
 Let $U=Y\cup\{x_{j1},\ldots,x_{j,t-1}\}$ be the additional block of $Z_j.$ If $T$ is a $(n+k-1)$-inset of $Z,$ then
 $T\cup\{x_{jt}\}$ is the  $(n+k)$-inset of $X$ not containing elements of $X_j,$ the  second index of which is greater
 than $t.$ The converse also holds.
The assertion follows by summing over $t,\;(1\leq t\leq q_j).$
Equation (\ref{krk}) is proved.

Omitting  the $j$th main block  of $X,$ we obtain a  set $Z.$ Each
$n+k$-inset of $X$ may be obtained as a union of some $n+k-i$-
inset of $Z,\;(1\leq i\leq q_j)$  and some of ${q_j\choose
i}$ $i$-subsets of the omitting main block, which proves (\ref{r3}).
\end{proof}

\section{Connections with Other Classes of Integers}
We noted that our function is closely connected with the binomial
coefficients. In this section, we establish its relation to
some other classes of integers.
\begin{proposition} If $n\geq 0,$ then  \[{n,\;2\choose n+2,\;3}=\frac{(n+5)(n+6)}{2},\]
that is,  ${n,\;2\choose n+2,\;3}$ equals the $(n+5)$th
\textbf{triangular number} {A000217}.
\end{proposition}
\begin{proof} The proof follows from (\ref{ndva}).
\end{proof}

\begin{proposition} If $n\geq 2,$ then
\[{n-1,\;1\choose 1,\;n}=\frac{3n(n-1)}{2},\]
that is, ${n-1,\;1\choose 1,\;n}$ equals the $(n-1)$th
\textbf{triangular  matchstick number} {A045943}.
\end{proposition}
\begin{proof} The proof follows from (\ref{sn=1}).

 We also give a proof in terms of insets.
Note first that ${1,1\choose 1,2}=3$ equals the first triangular
matchstick number. Denote $TM_n={n,\;1\choose 1,\;n+1}.$  We want to calculate the difference
$TM_n-TM_{n-1}.$ Consider two
sets $X$ and $Z,$  both having one main block. Let
$\{x_1,x_2,\ldots,x_n,x_{n+1}\},$ $\{x_1,x_2,\ldots,x_n\}$ be the main
blocks of $X$ and $Z$ respectively, and let
$\{y_1,y_2,\ldots,y_{n}\}$ and  $\{y_1,y_2,\ldots,y_{n-1}\}$ be
the additional blocks. The number $TM_n-TM_{n-1}$ equals the number of
$2$-insets of $X,$ which are not insets of $Z.$ Such an inset must
contain either  $x_{n+1}$ or $y_{n}.$ All  insets of this form are
\[\{x_i,y_{n}\},(i=1,2,\ldots,n+1),\},\;\{x_{n+1},y_i\},\;(i=1,2,\ldots,n-1)
\{x_i,x_{n+1}\},(i=1,2,\ldots,n),\]  which are
$3n$ in number. We conclude that \[TM_{n}-TM_{n-1}=3n,\] which is
the recurrence for the triangular  matchstick numbers.
\end{proof}

\begin{proposition} The following formula holds:
\[{n,1\choose 1,n}=\frac{n(3n-1)}{2},\] that is,
  ${n,1\choose 1,n}$ equals the $n$th \textbf{pentagonal
  number} {A000326}.
\end{proposition}
\begin{proof} Firstly, a $2$-inset may consist of pairs of elements from  the main block, and there is ${n\choose 2}$ such pairs. Secondly, it may consist of one element from the main and one element from the additional block. There is $n^2$ such insets. We thus have ${n\choose 2}+n^2=\frac{n(3n-1)}{2}$ $2$-insets.
\end{proof}
\begin{proposition} The following formula is true:
     \[{n,2\choose 1,n}=(2n-1)n^2,\] that is,
${n,2\choose 1,n}$ equals the $n$th
 \textbf{structured hexagonal prism number} {A015237}.
\end{proposition}

\begin{proposition}  If $Q=\{2,3\},$ then \[{m,2\choose 2,Q}=3(m+1)^2+2,\]
that is, ${m,2\choose 2,Q}$ equals the \textbf{number of points on
the surface of a square pyramid} {A005918}.
\end{proposition}
We give a short proof in terms of insets.
\begin{proof} Let $X_1=\{x_{11},x_{12}\},\;X_2=\{x_{21},x_{22},x_{23}\}$ be the main blocks,
 and $y=\{y_1,y_2,\ldots,y_n\}$ be the additional block of $X.$ In the next table we write different types of $4$-insets and its numbers.
\[\begin{array}{cc}\text{$4$-insets}&\text{its number}\\
\{x_{11},x_{12},x_{2i},x_{2j}\},&3,\\
\{x_{11},x_{12},x_{2i},y_j\},&3m,\\
\{x_{1i},x_{21},x_{22},x_{23}\},&2,\\
\{x_{1i},x_{2j},x_{2k},y_s\}&6m,\\
\{x_{1i},x_{2j},y_{k},y_s\},&6{m\choose 2}.
\end{array}\]
We have $3(m+1)^2+2$ $4$-insets in total.

\end{proof}

    It follows from (\ref{rek2}) that the numbers ${m,n\choose k,Q}$ form a Pascal-like
array, in which the  first row ($m=0$) begins with $q_1\cdot
q_2\cdots q_n.$

In the particular case  $n=1,$ the first row is \[{q\choose
1},{q\choose 2},\ldots,{q\choose q}.\] Hence, if $q=2,$  the first
row is $2,1$, so that we obtain the reverse Lucas triangle
{A029653}.  We note  one property of this triangle
connected with the \textbf{figurate numbers}. The third
column consists of $2$-dimensional square numbers, the
forth column consists of $3$-dimensional square
 numbers, and so on. We conclude from this that the following proposition is true:
\begin{proposition}For $m>0,\;k>2,$ the number  ${m,1\choose k,2}$ equals the $m$th
 $k$-dimensional \textbf{square pyramidal number} {A000330}.
 \end{proposition}
\begin{proof} The proof follows from the preceding notes. We also  give a short
 bijective proof. According to (\ref{ndva}),
we have \[{m,1\choose k,2}={m\choose k}+{m+1\choose k}.\] Let
$X_1=\{x_1,x_2\}$ be the main, and  $Y=\{y_1,y_2,\ldots,y_m\}$ the
additional block of $X$.
 Consider two disjoint  sets $A=\{a_1,a_2,\ldots,a_m\},\;B=\{b_1,b_2,\ldots,b_{m+1}\}.$
 Let the set $C$ consist of
  $k$-subsets of $A$ and $k$- subsets of $B.$  We need
to define a bijection between the set of $(k+1)$-insets of $X$ and
the set $C.$ A bijection goes as follows:
 \[\begin{array}{cc} 1.& \{x_1,x_2,y_{i_1},y_{i_2},\ldots,y_{i_{k-1}}\}
\leftrightarrow \{b_{i_1},b_{i_2},\ldots,b_{i_{k-1}},b_{m+1}\},\\
2.&\{x_1,y_{i_1},y_{i_2},\ldots,y_{i_{k}}\} \leftrightarrow
\{b_{i_1},b_{i_2},\ldots,b_{i_{k-1}},b_{i_k}\},\\
3.&\{x_2,y_{i_1},y_{i_2},\ldots,y_{i_{k}}\} \leftrightarrow
\{a_{i_1},a_{i_2},\ldots,a_{i_{k-1}},a_{i_k}\}. \end{array}\]
\end{proof}

The following result follows from the fact that, for $q=3,$ the
third column $1,4,10,19,31,\ldots$ of the above array consists of
the centered triangular numbers.
\begin{proposition} For $m>0,\;k>1,$ the number  ${m,1\choose k,3}$ equals the $(m+1)$th
$k$-dimensional \textbf{centered triangular number}
{A047010}.
\end{proposition}

For $q=4,$ the array consists of $m$-dimensional centered
tetrahedral numbers, and so on.
Hence,
\begin{proposition} For $m>0,\;k>1,$ the number  ${m,1\choose k,4}$ equals the $(m+1)$th $k$-dimensional \textbf{centered
tetrahedral  number} {A047030}.\end{proposition}

The fourth column ( omitting two first terms $1$ and $5$), in the
case $q=3,$ consists of numbers $15,34,65,111,\ldots,$ which are
of the form $\frac{m(m^2+1)}{2},\;(m=3,4,\ldots).$ This fact
connects our function with the \textbf{magic constants} {A006003}.
 \begin{proposition} For $m>2,$ the number ${m,1\choose 3,3}$ equals the
 magic constant for  the standard $m\times m$ magic
 square.
 \end{proposition}
\begin{proof} The proof follows from (\ref{ndva}).
We again  add  a short bijective proof. Let  $X_1=\{x_1,x_2,x_3\}$
be the main block of $X$, and let $Y=\{y_1,y_2,\ldots,y_m\}$ be
the additional block.  We have \[\frac{m(m^2+1)}{2}={m+1\choose
2}+m{m\choose 2}.\] Consider the following two sets:
$A=\{a_1,a_2,\ldots,a_{m+1}\}$ and $B=\{b_1,b_2,\ldots,b_m\}.$ Let
$C$ be the union of the set of   $2$-subsets of $A$ and
$\{iB_2\vert i\in\{1,2,\ldots,m\},$  where $B_2$ runs over all
$2$-subsets of $B$. We  define a bijection between sets $X$ and
$C$ in the following way:
 \[\begin{array}{cc}1.&\{x_1,x_2,x_3,y_i\}\leftrightarrow \{a_i,a_{m+1}\},\\
 2.&\{x_1,x_2,y_i,y_j\}\leftrightarrow \{a_{i},a_j\},\\
 3.&\{x_2,x_3,y_i,y_j\}\leftrightarrow i\{b_{i},b_j\},\\
 4.&\{x_1,x_3,y_i,y_j\}\leftrightarrow j\{b_{i},b_j\},\\
 5.&\{x_1,y_i,y_j,y_k\}\leftrightarrow i\{b_{j},b_k\},\\
6.&\{x_2,y_i,y_j,y_k\}\leftrightarrow
j\{b_{i},b_k\},\\
7.&\{x_3,y_i,y_j,y_k\}\leftrightarrow k\{b_{i},b_j\}.
\end{array}\] \end{proof}

Take  $Q=\{2,q\}.$  In this case, formula (\ref{ndva}) takes the
following form:
\[{m,2\choose
2,Q}=\frac{q^3}{3}+\bigg(m-\frac 12\bigg)q^2+\bigg(m^2-m+\frac
16\bigg)q.\] This easily implies that
\begin{equation}\label{tspn}{m,2\choose 2,Q}=m^2+(m+1)^2+\cdots+(m+q-1)^2.\end{equation}
\begin{proposition} The number ${m,2\choose 2,Q},$ where  $Q=\{2,q\},$ counts the
\textbf{truncated square pyramidal numbers} {A050409}.
  \end{proposition}

There is a relationship of our function with coefficients of the
Chebyshev polynomials of the second kind, which immediately
follows from (\ref{pr3}).
\begin{proposition}
 Let $c(n,k)$ denote the coefficient of $x^k$ of the \textbf{Chebysehev polynomial} $U_n(x)$
 {A008312}.  Then,
    \[c(n,k)=(-1)^{\frac{n-k}{2}}{0,\frac{n+k}{2}\choose \frac{n-k}{2},2},\] if
     $n$ and $k$ are of the same parity, otherwise $c(n,k)=0.$
\end{proposition}

\begin{rema} In Janji\'c paper \cite{MJ}, the preceding connection is used to define a generalization
 of the Chebyshev polynomials.
\end{rema}

We now establish a connection of our function to the
\textbf{Catalan numbers} {A000108}. Using (\ref{sn=1}), we
obtain
    ${2n,\;1\choose n,\;2}={2n+2\choose n+1}-{2n\choose n+1}=\frac{3n+2}{n+1}{2n\choose n}.$
Hence,
\begin{proposition} If $C_n$ is the $n$th  Catalan number, then
\[C_n=\frac{1}{3n+2}{2n,\;1\choose n,\;2}.\]
\end{proposition}
\begin{proof} Let $X=\{x_1,x_2\}$ be the main, and $Y=\{y_1,y_2,\ldots,y_{2n}\}$ be the additional block of $X.$
In the next table we write different types of $(n+1)$-insets of $X$ and its numbers.
\[\begin{array}{cc}\text{$(n+1)$-insets}&\text{its number}\\
\{x_{1},x_{2},y_{i_1},y_{i_2},\ldots,y_{i_{n-1}}\},&{2n\choose n-1},\\
\{x_{1},y_{i_1},y_{i_2},\ldots,y_{i_{n-1},y_{i_n}}\},&{2n\choose n},\\
\{x_{2},y_{i_1},y_{i_2},\ldots,y_{i_{n-1},y_{i_n}}\},&{2n\choose n}.
\end{array}\]
We thus have ${2n\choose n-1}+2{2n\choose n}=\frac{3n+2}{n+1}{2n\choose n}$ $(n+1)$-insets.

\end{proof}

\begin{proposition}
If $F_q$ is the Fibonacci number, and $Q=\{F_q,F_{q+1}\},$  then
\[{1,2\choose 1,Q}={q+2\choose 3}_F,\]
where ${q+2\choose 3}_F$ is the \textbf{Fibonomial coefficient}
{A001655}.
\end{proposition}
\begin{proof} The formula is an easy consequence of (\ref{2q}).
\end{proof}
Finally, we connect our function with Dalannoy and Sulanke
numbers.

The  \textbf{Delannoy number}  $D(m,n)$ {A008288} is
defined as the number of lattice paths from $(0,0)$ to $(m,n),$
using steps $(1,0),(0,1)$ and $(1,1).$

\begin{proposition} We have
\begin{equation}\label{dav}D(m,n)={m,n\choose n,2}.
\end{equation} \end{proposition}
\begin{proof} We
obviously have \[{0,n\choose n,2}={m,0\choose 0,2}=1.\]
Furthermore, for $m,n\not=0,$ using (\ref{rek2}), we obtain
\[{m,n\choose n,2}={m-1,n\choose n,2}+{m-1,n\choose n-1,2}.\] Applying
(\ref{krk}), we have \[{m-1,n\choose n-1,2}={m-1,n-1\choose
n-1,2}+{m,n-1\choose n-1,2}.\] It follows that \[{m,n\choose
n,2}={m-1,n\choose n,2}+{m-1,n-1\choose n-1,2}+{m,n-1\choose
n-1,2}.\] Hence, the numbers ${m,n\choose n,2}$ satisfy the same
recurrence relation as do the Dalannoy numbers. \end{proof}
\begin{rema} In his paper \cite{SU}, Sulanke gave the collection of 29 configurations  counted by the central Dallanoy numbers.
\end{rema}
The \textbf{Sulanke numbers}   $s_{n,k},\;(n,k\geq 0)$
{A064861} are defined in the following way:
\[s_{0,0}=1,\;\;s_{n,k}=0,\; \mbox{ if $n<0$ or $k<0$},\]
and
\[s_{n,k}=\begin{cases}s_{n,k-1}+s_{n-1,k}&\text{ if $n+k$ is even;}\\
s_{n,k-1}+2s_{n-1,k}&\text{ if $n+k$ is odd.}
\end{cases}\]

\begin{proposition} The following equations are true:
\begin{equation}\label{sun}
s_{n,k}=\begin{cases}{\frac{n+k}{2},\frac{n+k}{2}\choose k,\;2},&\text{if $n+k$ is even;}\\
{\frac{n+k-1}{2},\frac{n+k+1}{2}\choose k,\;2},&\text{if $n+k$ is
odd.}
\end{cases}
\end{equation}
\end{proposition}
\begin{proof}
 According to (\ref{rek2}), for even $n+k,$ we have
 \[{\frac{n+k}{2},\frac{n+k}{2}\choose k,\;2}=
 {\frac{n+k-2}{2},\frac{n+k}{2}\choose k-1,\;2}+
 {\frac{n+k-2}{2},\frac{n+k}{2}\choose k,\;2}.\]

For odd $n+k,$ using (\ref{r3}), we obtain
\[{\frac{n+k-1}{2},\frac{n+k+1}{2}\choose k,\;2}=
 {\frac{n+k-1}{2},\frac{n+k-1}{2}\choose k-1,\;2}
 +2{\frac{n+k-1}{2},\frac{n+k-1}{2}\choose k,\;2}.\]
We see that the numbers on the right side of (\ref{sun}) satisfy
the same recurrence as do the Sulanke numbers.
\end{proof}
Equation (\ref{q=2}) implies the following explicit formulae for
the Sulanke numbers:
\[s_{n,k}=\sum_{i=0}^{\frac{n+k}{2}}2^{\frac{n-k+2i}{2}}{\frac{n+k}{2}\choose i}
{\frac{n+k}{2}\choose k-i},\] if $n+k$ is even, and
\[s_{n,k}=\sum_{i=0}^{\frac{n+k-1}{2}}2^{\frac{n+1-k+2i}{2}}{\frac{n+k-1}{2}\choose i}
{\frac{n+k+1}{2}\choose k-i},\] if $n+k$ is odd.
\begin{rema} Using the method of $Z$ transform, J. Velasco, in his paper \cite{JV},  derived similar formulae for Sulanke numbers.
\end{rema}

The following two results connect our function with the
\textbf{coordination sequences} and the \textbf{crystal ball
sequences} for  cubic lattices.

\begin{proposition}
\begin{enumerate} \item The number
${m,n\choose n,2}$ equals the number of  solutions of the
Diophantine inequality \begin{equation}\label{nej}\vert
x_1\vert+\vert x_2\vert+\cdots+\vert x_n\vert\leq m.\end{equation}
\item The number ${m-1,n\choose n-1,2}$ equals the number of
solution of the Diophantine equation
\begin{equation}\label{jed}\vert x_1\vert+\vert x_2\vert+\cdots+\vert x_n\vert=m.\end{equation}
\end{enumerate}
\end{proposition}
\begin{proof}

      Each solution  $(a_1,a_2,\ldots,a_n)$ of (\ref{nej}) corresponds  to a
      $2n$-inset $T$
of $X$ as follows:

 If $a_i=0,$ then both elements of the main block $X_i$ are inserted in $T.$ If $a_i\not=0,$
and its sign is $+$, then the first element from $X_i$ is inserted
into $T.$ If the sign of $a_i$ is $-$, then the second
 element of $X_i$ is inserted into $T.$
 In this way, we insert elements from the main blocks into $T.$

 Assume that  $X_{i_1},X_{i_2},\ldots,X_{i_t},\;
1< i_1<i_2<\ldots<i_t,\;(1\leq t\leq n)$ are the main blocks from
which, up until now,  only one element is inserted into $T.$ This means
that $a_{i_1},a_{i_2},\ldots,a_{i_t}$ are all different from $0.$
Also, $|a_{i_1}|+|a_{i_2}|+\cdots+|a_{i_t}|\leq m.$ Now, we insert
elements
\[y_{|a_{i_1}|},y_{|a_{i_1}|+|a_{i_2}|},\ldots,y_{|a_{i_1}|+|a_{i_2}|+\cdots+|a_{i_t}|}\] from the additional
block $Y$ into $T.$ In this way, we obtain a $2n$-inset $T.$

Now, we have to prove that this correspondence is bijective.

Let $T$ be an arbitrary $2n$-inset of $X.$ If there are no elements
of $Y$ in $T,$ then $T$ is obtained by the trivial solutions of
 (\ref{nej}). Assume  that $T$ contains the
 subset $\{y_{i_1},y_{i_2},\ldots,y_{i_s}\},\;(1\leq i_1<i_2<\cdots<i_s\leq m)$ of $Y.$
   We also have $s\leq n,$ since a $2n$-inset of
  $X$ has at most $n$
  elements from the additional block $Y.$

Form the solution $(b_1,b_2,\ldots,b_n)$ of (\ref{nej}) in the
following way:  Since there are $s-n$ main blocks $X_t$ from which
both elements are in $T,$ we define $b_t=0.$ Let
$X_{u_1},X_{u_2},\ldots,X_{u_s}$ be the remaining main blocks.
 We define $|b_{u_1}|=i_1,$ and the sign of $b_{u_1}$ is $+,$
 if the first element of the main block $X_u$ is in $T,$ and the sign $-$
 otherwise. We next define $|b_{u_t}|=i_{u_t}-i_{u_{t-1}},\;(t=2,\ldots,s),$ choosing the sign of
 $b_{u_t}$ in the same way as for $b_{u_1}.$ It follows that
 $|b_{u_1}|+\cdots+|b_{u_s}|=i_s\leq m.$ Hence,
 $(b_1,b_2,\ldots,b_n)$ is the solution of (\ref{nej}), which in
 the preceding correspondence produces the inset $T.$
  This means that the correspondence is
  surjective.

It is clear that no two different solutions may produce the same
inset, which means that our correspondence is injective. This
proves (\ref{nej}).

Using (\ref{rek2}), we have \[{m-1,n\choose n-1,2}={m,n\choose
n,2}-{m-1,n\choose n,2},\]
 which proves (\ref{jed}).
\end{proof}
\begin{rema}
Note that the number of solutions of equation (\ref{jed})
is the number of the coordination sequence, and the solution of
(\ref{nej}) are the numbers of the crystal ball sequence for the
cubic lattice $\mathbb Z^n.$  Also, the number of  solutions of
(\ref{nej}) equals the Dalannoy number $D(m,n).$
 \end{rema}
\begin{rema} The formulae (\ref{nej}) and (\ref{jed}) concern the
following sequences in OEIS \cite{slo}: {A001105},
{A035597}, {A035598}, {A035599},
{A035600}, {A035601}, {A035602},
{A035603}, {A035604}, {A035605},
{A035605}.
\end{rema}

Comparing the results of the preceding proposition, and the
formulae (16) and (17) in
 Conway and Sloane  \cite{slo1}, we obtain the following binomial identities:
\[\sum_{i=0}^m2^i{m\choose i}{n\choose i}=
\sum_{i=0}^n{n\choose i}{m-i+n\choose n},\]

\[\sum_{i=0}^{m-1}2^{i+1}{m-1\choose i}{n\choose i+1}=
\sum_{i=0}^n{n\choose i}{m-i+n-1\choose n-1}.\]

\section{Some Configurations Counted by ${m,n\choose k, Q}.$}
In this section, we describe a number of configurations counted by
our function.
The first result concerns the complete bipartite graphs.
\begin{proposition} The number $M(n,q-1,q)$ equals the number of spanning subgraphs of the complete bipartite graph
$K(q,n),$ having $n+q-1$ edges with no isolated vertices.
\end{proposition}
 \begin{proof} Let $A=(a_{ij})_{n\times n}$ be $(0,1)$-matrix
  which has $n+q-1$ ones, and has no zero rows or zero columns.
  This matrix corresponds to  a spanning subgraph $S=(V(S),E(S))$ of the complete bipartite graph
 $K(n,q)=(V=\{v_1,v_2,\ldots,v_n\} \cup \{u_1,u_2,\ldots,u_q\},E=\{v_iu_j : 1 \leq i \leq n,\; 1 \leq j \leq q\}
),$ having $n+q-1$ edges, in the following way:
\begin{enumerate} \item  $a(i,j)=1,\;(1 \leq i \leq n,\; 1 \leq j \leq q )$ if and only if $v_iu_j \in
E(S)$. \item  $a(i,j)=0,\;(1 \leq i \leq n,\; 1 \leq j \leq q),$
if and only if   $v_iu_j \notin E(S)$.
\end{enumerate}
Note that the matrix $A$ has
$n+q-1$ ones if and only if $|E(S)|=n+q-1,$
and that the matrix
$A$ has no zero rows or zero columns if and only if the subgraph  $S$ has no
isolated vertices.
 \end{proof}
\begin{rema} The function $M(n,q-1,q)$ produces the following sequences in \cite{slo}:
 {A001787}, {A084485}, {A084486}.
\end{rema}
\begin{proposition}
 The number ${n,\;1\choose n-1,n}$ equals the number
 of square submatrices of some $n$ by $n$ matrix {A030662}.
 \end{proposition}
\begin{proof} Let $M$ be a square matrix of order $n.$
If $X_1=\{x_1,\ldots,x_n\}$ is the main block of $X,$ and
$Y=\{y_1,\ldots,y_n\}$ is the additional block, then each
$n$-inset of $X$ has the form
\[\{x_{i_1},x_{i_2},\ldots,x_{i_k},y_{j_{k+1}},\ldots,y_{j_n}\},\;(k\geq 1).\]
Every  such  inset corresponds to the  square
submatrix of $M,$ of which the indices of rows  are $i_1,i_2,\ldots,i_k,$
and indices of columns  belong to the set
$\{1,2,\ldots,n\}\setminus\{j_{k+1},\ldots,{j_n}\}.$\end{proof}

\begin{proposition} For $i\geq 0,$ the number ${n,\;1\choose n+i-1,\;n}$ equals the
number of lattice paths from $(0,0)$ to $(n,n),$ with steps
$E=(1,0)$ and $N=(0,1),$ which either touch or cross the line $x-y=i.$
\end{proposition}
\begin{proof}
\end{proof}
We may write arbitrary lattice path from $(0,0)$ to $(n,n)$ in the form
$P=P_1P_2\ldots P_{2n},$ where each $P_i$ is either $E$ or $N.$ Assume
that $s$ is the least index such that the end of $P_s$ touches the line $x-y=i,$
 and let  $(r,r-i), \;(i\leq r\leq n)$ be the touching point. It follows that $s=2r-i.$

Consider the lattice  path $Q=Q_1Q_2\ldots Q_sP_{s+1}\ldots P_{2n},$ where $P_t$ and $Q_t$ are symmetric with respect to the line $y=x-i.$ This path connects $(-i,i)$ and $(n,n).$ Since every lattice path from $(i,-i)$ to $(n,n)$ must cross the line $y=x-i,$ conversely also holds.
We thus have a bijection between the number of considered lattice  paths and the number of all
 lattice paths from $(i,-i)$ to $(n,n).$ The last lattice paths are of the form
 $L_1L_2\ldots L_{2n},$ where $n-i$ $L$'s equal $E,$ and $n+i$ equal $N.$
Hence, its number is ${2n\choose n+i}$, and the proof follows from (\ref{ndva}).

Again, we add a short bijective proof. Let $X$ be a set which have one main block
 $X_1=\{x_1,x_2,\ldots,x_n\},$ and the additional block $Y=\{y_1,y_2,\ldots,y_n\}.$
We need to define a bijection between all lattice paths from $(i,-i)$ to $(n,n)$ and $(n+i)$-insets of of $X.$
Let $\{x_{i_1},\ldots,x_{i_u},y_{j_1},\ldots,y_{j_v},\}\;(u+v=n+i)$ be an $(n+i)$-inset of $X.$
Define the path  $L_1L_2\ldots L_{2n}$ in the following way.
\[L_p=\begin{cases}N&\text{if $p\in\{i_1,\ldots,i_u,j_1,\ldots,j_v\}$},\\
E&\text{otherwise}.\end{cases}\]
It is clear that this correspondence is bijective. 
\begin{rema} This proposition concerns the following sequences in OEIS \cite{slo}:
 {A001791}, {A002694}, {A004310} ,{A004311}
 {A004312}, {A004313}, {A004314}, {A004315}, {A004316}, {A004317}, {A004318}.
\end{rema}

We now give a combinatorial interpretation of the formula (\ref{tspn}).

Consider the square $Q,$  the vertices of which are  $(1,1),(1,m+q),(m+q,1),$ and $(m+q,m+q).$
Let $S$ be the set of squares, whose vertices are
$(u,v),(u+w,v),(u,v+w),(u+w,v+w),\;(1\leq w\leq q),$
and which are contained in $Q.$
\begin{proposition}
If $Q=\{2,q\},$ then the
number ${m,2\choose 2,Q}$   equals  $|S|.$
\end{proposition}
\begin{proof}
Let $X_1=\{x_{11},x_{12}\}$  and
$X_2=\{x_{21},x_{22},\ldots,x_{2q}\}$ be the main blocks of $X,$
and let $y=\{y_1,\ldots,y_m\}$ be the additional block.

We need to define a bijection between $4$-insets of $X$ and the set $S.$
      If $U$ is a $4$-inset of $X,$ then it must contain an element from $X_2.$ The length of the side of corresponding square will be the minimal $i,$ such that $x_{2i}\in U.$
      In the next correspondence, it will be denoted by $d.$
We now define a correspondence between $4$-insets and pairs $(i,j),$ which represent
the upper right corner of the square. Note that the indices of elements in insets are always taken in increasing order.
\[\begin{array}{cc}
1.&\{x_{11},x_{2d},x_{2i},x_{2j}\}\leftrightarrow
(i,j),\\
2.&\{x_{12},x_{2d},x_{2i},x_{2j}\}\leftrightarrow
(j,i),\\
 3.&\{x_{11},x_{12},x_{2d},x_{2i}\}\leftrightarrow
(i,i),\\
4.&\{x_{11},x_{12},x_{2d},y_{i}\}\leftrightarrow
(q+i,q+i),\\
5.&\{x_{11},x_{2d},x_{2i},y_{j}\}\leftrightarrow
(i,q+j),\\
6.&\{x_{12},x_{2d},x_{2i},y_{j}\}\leftrightarrow
(j+q,i),\\
7.&\{x_{11},x_{2d},y_i,y_j\}\leftrightarrow
(q+i,q+j),\\
8.&\{x_{12},x_{2d},y_i,y_j\}\leftrightarrow
(q+j,q+i).
\end{array}\]
It is easy to see that the correspondence is bijective.
 \end{proof}
\begin{proposition} Let $p_1<p_2<p_3$ be prime numbers. If we denote $s=p_1p_2p_3^2,$ then
 $n,2\choose 1,n$ equals the number of divisors of $s^{n-1}$
{A015237}.
\end{proposition}
\begin{proof} Let $X_{i}=\{x_{i1},x_{i2},\ldots,x_{in}\},\;(i=1,2)$ be the main blocks of $X$,
and let $Y=\{y_1,y_2,\ldots,y_n\}$ be the additional block. It is
enough to define a bijection between $3$-insets of $X,$ and
$3$-tuples $(i,j,k),$ such that $0\leq i,j\leq n-1,\;0\leq k\leq
2n-2.$ A bijection goes as follows:
\[\begin{array}{cc}
1.&\{x_{1i},x_{1j},x_{2k}\}\leftrightarrow
(i-1,j-1,k-1),\\
2.&\{x_{1k},x_{2j},x_{2i}\}\leftrightarrow (i-1,j-1,k-1),\\
3.&\{x_{1i},x_{2j},y_{k}\}\leftrightarrow
(i-1,j-1,n+k-2),\;(1<k).\\
4.&\{x_{1i},x_{2j},y_1\}\leftrightarrow (i-1,i-1,j-1).\end{array}\]
\end{proof}

\begin{proposition}
The number ${1,n\choose 1,2}$ equals the number of parts in all
compositions  of   $n+1$ {A001792}.
\end{proposition}
\begin{proof} Let $X_i=\{x_{i1},x_{i2}\},\;(i=1,2,\ldots,n)$ be the main blocks of $X,$
and let $Y=\{y\}$ be the additional block. For a fixed
$k,\;(k=0,1,\ldots,n),$ we shall prove that $(n+1)$-insets of $X,$
in which exactly  $k$ elements of the form $x_{i1}$ are not chosen,
count the number of parts in all compositions of $n+1$ into
$n-k+1$ parts.
 Take $1\leq i_1<i_2<\ldots<i_k\leq n,$ and consider $(n+1)$-inset $U$ of $X$ not containing
 elements $x_{i_1,1},x_{i_2,1},\ldots,x_{i_k,1},$ but containing the remaining $n-k$
 elements of the forme $x_{i1}.$
 The remaining $k$ of $k+1$ elements of $U$ must be $x_{i_1,2},x_{i_2,2},\ldots,x_{i_k,2}.$
 For the remaining element, therefore, either $y$ or one of
 $x_{j2},\;(j\not=i_t,\;(t=1,\ldots,k))$ must be chosen.  For this, we have $(n-k+1)$ possibilities.
 Since $i_1,\ldots,i_k$ may be chosen in ${n\choose k}$ ways, we have
 $(n-k+1){n\choose n-k}$ insets containing   $(n+1)$ elements, but not containing
 exactly $k$ elements of the form $x_{i1}.$ On the other hand,
 the number of the compositions of $n+1$ with $n-k+1$ parts equals
 ${n\choose n-k}.$ Hence,
 $(n-k+1){n\choose n-k}$  equals the number of parts in all
 compositions of $n+1$ with $(n-k+1)$ parts. Since $k$ ranges from $0$ to $n,$
 the assertion follows.
  \end{proof}

We now  present two configurations counted by the number
${1,n-1\choose 1,\;3},\;(n>1)$ {A027471}.

 \begin{proposition}

  Given $n$ points on a straight line,
  the number ${1,n-1\choose 1,\;3}$ equals the number of coloring of $n-1$ points
   with three colors.
  \end{proposition}
\begin{proof}
Let $X_i=\{x_{i1},x_{i2},x_{i3}\},\;(i=1,\ldots,n-1)$ be the main
blocks of $X,$ and $Y=\{y\}$ be the additional block.  We define a
correspondence between $n$-insets of $X$ and the above-defined
colorings in the following way:
\begin{enumerate}
\item If $U$ is an $n$-inset such that $y \in U,$ then $U$ contains
exactly one element from each of the main blocks. If $x_{ij}\in
U,$ then we color the point $i$ by the color $j.$
     In this way, the point $n$ remains uncolored.

\item If $y \notin U$, then there is exactly one main block $k,$
two elements of which are in $U.$  In this case, the $k$th point
remains uncolored. If $x_{km}\not\in U,$ then the point $n$ is
colored by the color $m.$  If $x_{ij}\in U,\;(i\not=k),$ then we
color the point $i$ by the color $j.$

The correspondence is clearly bijective.
\end{enumerate}

\end{proof}

\begin{proposition}
Assume $n>1.$ Then,
  \begin{equation}\label{dif}{1,n-1\choose 1,\;3}=\sum_{X\subseteq Y\subseteq [n]}(|Y|-|X|).\end{equation}
\end{proposition}
\begin{proof}
Take $k,$ such that $1\leq k\leq n.$ We count all pairs $X\subseteq
Y\subseteq [n],$ such that $|Y|-|X|=k.$ If
$X_k=\{x_1,x_2,\ldots,x_k\}$ is a given $k$-subset of $[n],$ and
if  $Z_k$ is arbitrary subset of
$[n]\setminus\{x_1,x_2,\ldots,x_k\},$ ($\emptyset$ included), then
$|X_k\cup Z_k|-|Z_k|=k.$ Hence, for a fixed $X_k$ there are
$2^{n-k}$ mutually different $Z_k$'s. On the other hand, there are
${n\choose k}$ mutually different $X_k$'s. We conclude that there
are ${n\choose k}2^{n-k}$ pairs $(U,V)$ of subsets, where $U-V$ has
 $k$ elements. The sum on the right side of (\ref{dif}) thus equals
$\sum_{k=1}^nk{n\choose k}2^{n-k}.$ It is easy to see that
\[\sum_{k=1}^nk{n\choose k}2^{n-k}=n3^{n-1}.\] On the other hand,
if $X_{i}=\{x_{i1},x_{i2},x_{i3}\},\;(i=1,2,\ldots,n-1)$ are the
main blocks, and $Y=\{y\}$ the additional block of $X,$ then there
are obviously $3^{n-1}$ of $n$-insets of $X$ containing $y.$ The
$n$-insets of $X,$ not containing $y,$ must contain two elements
from one main block, and one element from the remaining main
blocks. For this, we have $3(n-1)3^{n-2}$ possibilities. Hence,
there are $3^{n-1}+3(n-1)3^{n-2}=n3^{n-1}$ $n$-insets of $X.$
\end{proof}

We next prove that our function, in one particular case, counts the
number of the so-called weak compositions. We let $c(n)$ denote the
number of the compositions of $n.$ It is well-known that
$c(n)=2^{n-1},\;(n>0).$ Additionally, we put $c(0)=1.$
Compositions in which some parts may be zero are called  weak
compositions.  We let $cw(r,s)$ denote the number of the weak
compositions of $r$ in which $s$ parts equal zero.

\begin{proposition}The following formula is true:
\begin{equation}\label{wc1}cw(r,s)=\sum_{j_1+j_2+\cdots+j_{s+1}=r}c(j_1)c(j_2)\cdots c(j_{s+1}),\end{equation}
where the sum is taken over  $j_t\geq 0,\;(t=1,2,\ldots,s+1).$
\end{proposition}
\begin{proof} We use induction with
respect to $s.$ For $s=0,$ the assertion is obvious. Assume that
the assertion is true for $s-1.$ Using the induction hypothesis,
we may write  equation (\ref{wc1}) in the following form:
\begin{equation}\label{wc}cw(r,s)=\sum_{j=0}^nc(j)cw(r-j,s-1).\end{equation}
Let $(i_1,i_2,\ldots,)$ be a weak composition of $r,$ in which
exactly $s$ parts equal $0.$  Assume that $i_p$ is the first part
  equal to zero. Then, $(i_1,\ldots,i_{p-1})$ is a composition
of $i_1+\cdots+i_{p-1}=j$ without zeroes. Note that $j$ can be
zero. Furthermore, $(i_{p+1},\cdots)$ is a weak composition of
$r-j$ with $s-1$ zeroes. For a fixed $j,$ there are
$c(j)cw(r-j,s-1)$ such compositions. Changing $j,$ we conclude
that the right side of (\ref{wc}) counts all weak compositions.
\end{proof}

\begin{proposition} Let $r,s$ be positive integers. Then,
 \begin{equation}\label{slko}cw(r,s)={s+1,n-1\choose s,\;2}.\end{equation}
 \end{proposition}
 \begin{proof}

 Collecting terms in (\ref{wc1}), in which the indices  $j_t$ equal zero, we obtain
\[cw(r,s)=\sum_{i=0}^s{s+1\choose i}\sum_{j_1+j_2+\cdots+j_{s-i+1}=r}
2^{j_1-1}2^{j_2-1}\cdots 2^{j_{s-i+1}-1},\] where the sum is taken
over $j_t\geq 1.$ Hence,
\[cw(r,s)=2^{r-s-1}\sum_{i=0}^s2^{i}{s+1\choose
i}\sum_{j_1+j_2+\cdots+j_{s-i+1}=r}1.\] Since the last sum is
taken over all compositions of $r$ with $s-i+1$ parts, we finally
have \[cw(r,s)=2^{r-s-1}\sum_{i=0}^{s+1}2^{i}{s+1\choose
i}{r-1\choose s-i},\] and the proof follows from (\ref{q=2}).
\end{proof}
\begin{rema} The formula (\ref{slko}) produces the following sequences in OEIS \cite{slo}
 {A000297}, {A058396}, {A062109}, {A169792}, {A169793}, {A169794},{A169795}, {A169796}, {A169797}.
 \end{rema}

We conclude the paper with three chessboard combinatorial
problems.

\begin{proposition} The number \[{n-1,\;2\choose 1, n}\]
equals the number of possible rook moves on an $n \times
n$ chessboard {A035006}.
\end{proposition}

\begin{proof}
Let $X_i=\{x_{i1},\ldots,x_{i,n}\},\;(i=1,2)$ be the main blocks
of $X,$ and $Y=\{y_1,y_2,\ldots,y_{n-1}\}$ be the additional block.
 We need a bijection of  $3$-insets of the set $X,$
  and all the possible rook moves on an $n\times n$ chessboard.
 The correspondence goes as follows:
 \[\begin{array}{cc}
 1.&\{x_{1i},x_{1j},x_{2k}\},\leftrightarrow
  [(i,k) \rightarrow (j,k)],\\
2.&\{x_{1k},x_{2i},x_{2j}\},\leftrightarrow
  [(j,k) \rightarrow (i,k)],\\
3.&\{x_{1i},x_{2j},y_{k}\},\leftrightarrow
  [(i,k) \rightarrow (i,j)],\;(j\not=k),\\
4.&\{x_{1i},x_{2j},y_{j}\},\leftrightarrow
  [(i,n) \rightarrow (i,j)],\;(j=k).\end{array}\]

According to (\ref{ndva}), the number of possible moves equals
$2(n-1)n^2.$
\end{proof}

\begin{proposition} If $n\geq 2,$ then the number \[{1,n\choose n-2,2}\] equals
 the total number of possible bishop moves on an $n\times n$ chessboard
 {A002492}.
 \end{proposition}
\begin{proof} We give two proofs.
\begin{enumerate}
\item This proof is bijective.

It is enough to count the number of moves from the field $(i,j)$ to
the field
 $(i+k,j+k),$ for a positive $k,$ such that $i+k\leq n,\;j+k\leq n.$ If $N$ is the number of such moves,
  then $4N$ is the number of all possible moves.

 Let set $X$ consists of $n$ main blocks $X_i=\{x_{i,1},x_{i,2}\},\;(i=1,2,\ldots,n),$
 and the additional block $Y=\{y\}.$
  We define a bijective correspondence between the set of  moves described above and one fourth of all $(n-2)$-insets of $X.$
 In fact, we define a bijection between the moves and the complements of $(n-2)$-insets
 of $X.$
 The complements are $3$-sets $\{a,b,c\}$ of $X,$ such that no two of its elements
 can be in the same main block.
The correspondence goes as  follows:
\begin{enumerate}
\item
$\{x_{i,1},x_{j,1},x_{k,1}\}\leftrightarrow [(i,j)\rightarrow(i+k-j,k)],\;(1\leq i<j<k).$
In this correspondence we have ${n\choose 3}$ elements.
\item   $\{x_{i,2},x_{j,2},x_{k,2}\}\leftrightarrow
[(j,i)\rightarrow(k,i+k-j),\;( i<j<k)].$
In this correspondence we also have ${n\choose 3}$ elements.
\item  $\{x_{i,1},x_{j,2},y\}\leftrightarrow [(i,i)\rightarrow(j,j)],\;(i<j).$
Now, we have ${n\choose 2}$ moves.
\end{enumerate}

It is clear that all moves are counted. For this we need
\[2{n\choose 3}+{n\choose 2}=\frac{n(2n-1)(n-1)}{6}\] insets. On the other hand, according to (\ref{q=2}), we have
\[{1,\;n\choose n-2,2}=\frac{2n(2n-1)(n-1)}{3},\] which proves the
assertion.

\item
 We let $T_n$ be an $n\times n$ chessboard, and let $a_n$ denote the total number of possible bishop
moves. We may consider that $T_{n+1}$ is obtained by adding to
$T_n$ one row at the top, and one column at the right. We
calculate $a_{n+1}-a_n,$ which is the number of moves on $T_{n+1}$
that are not possible on $T_n.$
\begin{enumerate} \item Firstly, if the bishop is on the main diagonal of $T_n$ or below, then only one additional  move is produced. We have thus obtained $\frac{n(n+1)}{2}$ new moves.
\item
For the bishop on $T_n,$ and above the main diagonal, there are
$3$ additional moves, or $3\frac{n(n-1)}{2}$ additional moves in total. \item
For each bishop on $T_{n+1}$ which is not on $T_n,$ we have $n$
additional moves. Hence, we have $n(2n+1)$ additional
    moves in total.
\end{enumerate}

We thus have $\frac{n(n+1)}{2}+3\frac{n(n-1)}{2}+n(2n+1)=4n^2$ additional  moves in total. Hence, the following recurrence is obtained:
\[a_{n+1}-a_n=4n^2.\] It is easy to see that ${1,\;n\choose n-2,2}$
satisfies this recurrence.

\end{enumerate}
\end{proof}

Since the queen can move both as a rook and  as a bishop,
we have \begin{proposition} The number \[{1,n\choose
n-2,2}+{n-1,\;2\choose 1,\; n},\;(n\geq 2)\]
 equals  the  possible queen moves on an $n \times n$ chessboard.
 This number  is $\frac{2n(5n-1)(n-1)}{3}$ {A035005}.
 \end{proposition}

Finally, we give a number of additional
configurations, counted by our function, and  described in
sequences in OEIS \cite{slo}.

\begin{center}
\begin{tabular}{cc} Function&Numbers of sequences\\
${0,n\choose k,2}$ &\begin{tabular}{c}{A000918}, {A001787}, {A001788},
{A001789},  {A003472},\\{A054849}, {A002409}, {A054851},
{A140325}, {A140354}, {A172242}\end{tabular}\\\hline\\
${1,n\choose 1,Q}$ &\begin{tabular}{c}{A059270}, {A094952}, {A069072},
 {A007531}, {A000466}, {A019583}, {A076301}\end{tabular}
 \\\hline\\
${m,1\choose k,q}$&\begin{tabular}{c}
{A015237}, {A160378}, {A027620}, {A028347}, {A028560},\\
{A034428}, {A000567}, {A045944}, {A123865}, {A034828}
\end{tabular}\\\hline\\
${m,2\choose k,Q}$&\begin{tabular}{c}{A080838},
{A015237}, {A091361}, {A017593}, {A063488}\\
{A039623}, {A116882}, {A081266},
{A202804}, {A194715}
\end{tabular}\\\hline\\
${m,n\choose k,2}$&\begin{tabular}{c}{A002002},
{A049600},
{A142978}, {A099776}, {A014820},\\
{A069039}, {A099195}, {A006325}, {A061927}, {A191596}\\
 {A001792}, {A045623}, {A045891},  {A034007}, {A111297}\\
 {A159694}, {A001788}, {A049611}, {A058396},
 {A158920}
\end{tabular}\\\hline\\

\end{tabular}

\end{center}

\end{document}